\documentclass[reqno,12pt]{amsart}
\usepackage{amssymb,amsmath,amsthm,array}
\usepackage{cite}
\usepackage{mathdots}
\usepackage{dsfont}
\usepackage{graphicx}
\usepackage{float}
\usepackage{caption}
\usepackage{subcaption}

\def\Z{\mathbb{Z}}

\def\R{\mathbb{R}}

\def\I{\mathcal{I}}

\DeclareMathOperator{\im}{Im}
\DeclareMathOperator{\dist}{dist}
\DeclareMathOperator{\spn}{span}
\DeclareMathOperator{\spt}{spt}

\newtheorem{theorem}{Theorem}
\newtheorem{lemma}[theorem]{Lemma}

\theoremstyle{remark}

\newtheorem*{example}{Example}

\numberwithin{equation}{section}

\begin{document}


\title{Topological states in the Kuramoto Model}
\author{Timothy Ferguson}

\begin{abstract}
The Kuramoto model is a system of nonlinear differential equations that models networks of coupled oscillators and is often used to study synchronization among the oscillators. In this paper we study steady state solutions of the Kuramoto model by assigning to each steady state a tuple of integers which records how the state twists around the cycles in the network. We then use this new classification of steady states to obtain a "Weyl" type of asymptotic estimate for the number of steady states as the number of oscillators becomes arbitrarily large while preserving the cycle structure. We further show how this asymptotic estimate can be maximized, and as a special case we obtain an asymptotic lower bound for the number of stable steady states of the model.
\end{abstract}

\maketitle

\smallskip
\noindent \footnotesize\textbf{Keywords.} Kuramoto model, fixed points, graph topology, Weyl asymptotic
\\
\\
\smallskip
\noindent \footnotesize\textbf{AMS subject classifications.} 34C15, 34D06



\section{Introduction}

The Kuramoto model is used to describe the behavior of a network of coupled oscillators. If we let $\theta_i$ denote the position of the $i$th oscillator, $\omega_i$ its natural frequency, and $\gamma_{ij}$ the strengh of the coupling between the $i$th and $j$th oscillators, then the Kuramoto model on a network of $N$ coupled oscillators is given by the system of $N$ nonlinear differential equations
\begin{align} \label{model}
\frac{d\theta_i}{dt} = \omega_i - \sum_{j=1}^N \gamma_{ij} \sin(\theta_i - \theta_j) \quad \text{for} \quad i \in \{1,\dots,N\}.
\end{align}
This model was first proposed by Kuramoto in \cite{MR762432} to study the phenomenon of synchronization. There are many interesting systems exhibiting synchronization including, for example, the flashing of interacting fireflies \cite{Buck1988SynchronousRhythmicFlashing}, the phase synchronization of arrays of lasers \cite{PhysRevA.52.4089}, and the synchronization of pacemaker cells in the heart \cite{peskin75}. For a review of the history of the Kuramoto model and a more complete list of synchronization phenomena see \cite{MR1783382}.

Unless stated otherwise we will work with the simplified model
\begin{align} \label{simplified model}
\frac{d\theta_i}{dt} = - \sum_{j=1}^N \gamma_{ij} \sin(\theta_i - \theta_j) \quad \text{for} \quad i \in \{1,\dots,N\}
\end{align}
obtained from $\eqref{model}$ by setting $\omega_i = 0$. Since $\eqref{simplified model}$ is a gradient system on the compact $n$-torus, we see that every solution converges to a steady state solution and for this reason steady state solutions are of key interest. (Note that although $\eqref{model}$ is also a gradient system the phase space is not compact.) A particular class of steady state solutions which has received a lot attention \cite{MR2220552,MR3388491,MR3129708,MR3415044,MR3600363} are the so called twisted states. These are steady states for which the positions of successive oscillators wrap around the unit circle. For example, if we consider a cyclic network with unit edge weights, then one can check that $\theta_i = \frac{2\pi qi}{N} \pmod{1}$ is a steady state solution of $\eqref{simplified model}$ for any integer $q$. This solution wraps around the unit circle $q$ times and is referred to as a $q$-twisted state. These steady states were first introduced by Wiley, Strogatz, and Girvan in \cite{MR2220552}. Since then their existence and stability on various networks and under various perturbations has received a lot of attention. After introducing $q$-twisted states in \cite{MR2220552} they studied their stability properties in the continuum Kuramoto model for $k$-nearest neighbor graphs. They found a sufficient condition for linear stability. In \cite{MR3388491}, Girnyk, Hasler, and Maistrenko again studied $q$-twisted states for $k$-nearest neighbor graphs. They showed that if a $q$-twisted state is stable for the continuum model then it is for the finite model for sufficiently large $N$. Furthermore they numerically investigated a new class of steady state solutions which they called multi-twisted states. These are states for which angle differences between neighbors are not restricted to only $2\pi q/N$. On two different sectors of the network, the angle differences between neighbors are allowed to be close to $2\pi q/N$ and $-2\pi q/N$ respectively with intermediate values in between. They further numerically showed that the number of such states appears to grow exponentially in $N$. In \cite{MR3129708}, Medvedev showed the existence of $q$-twisted states in small world graphs. He further showed that random long-range connections tend to increase the chance of synchronization while simultaneously decreasing the chance of stability. He did this by showing that as he increased a parameter, which statistically would result in more long range connections, that the chance of synchronization went up while the chance of any given $q$-twisted state being stable went down. In \cite{MR3415044}, Medvedev and Tang investigated the existence and stability properties of twisted states for Cayley graphs. In particular they showed that there exist Cayley graphs with similar properties, such as the distribution of edges, which exhibit different stability for $q$-twisted states. Finally, in \cite{MR3600363}, Medvedev and Wright showed that linear stability in the continuum model implies stability under perturbations by weakly differentiable and bounded variation periodic functions.

In this paper we investigate topological states for an arbitrary network. As a special case this will allow us to study the multiplicity of stable topological states. In section $\ref{main}$ we make definitions and state our main results without proof. We leave the proofs for an appendix at the end. For our first result we demonstrate a one-to-one correspondence between steady states of $\eqref{model}$ and the lattice points in a certain set, $W(A)$, depending on the underlying network. Here both $W(A)$ and the lattice have dimension $c = |E| - |V| + 1$ which arises from the cycle structure of the graph. Furthermore each component of a lattice point in the set records how the corresponding steady state twists around a particular cycle in the graph. This is the motivation for referring to the steady states of $\eqref{model}$ as topological states. It is not surprising that the number of steady states should depend on the cycle structure. For example, consider the continuum Kuramoto model on a figure eight quantum graph. In this case a steady state corresponds a local minimum of a certain functional. Since a solution is a continuous function on each loop, each loop must have at least one local minimum. However these loops are not homotopic so generically we expect that their local minimums are distinct. Thus we expect to have a local minimum for each loop which in our example is two.

In our second result we look at the how the set $A$ depends on the network. In our third result we show that if we allow the number of vertices in our network to increase while preserving the network structure, that the number of steady states has a "Weyl" type asymptotic estimate $|W(A)| N^c$. This result is surprising since this would suggest that the number of steady states increases as we add more cycles, equivalently edges, to the network. However Taylor showed in \cite{MR2878025} that the complete graph has only the trivial stable steady state $\theta_i = 0$. Therefore we see that either the asymptotics of the number of stable steady states with respect to the number of vertices and edges does not commute or that $|W(A)|$ must decrease rapidly as $c$ increases.

In our fourth and last main result we show how to maximize $|W(A)|$. We finish by applying our results to networks with one and two cycles.



\section{Definitions and main results} \label{main}

In this section we make definitions that will be used throughout the rest of the paper. Let $G = (V,E,\Gamma)$ be a simple weighted graph. We will assume for simplicity that for every edge $e$ there exists a spanning tree $T$ which does not contain $e$. In other words we don't have a graph with an attached tree. We will let $\spt(G)$ denote the set of spanning trees of $G$. Given such a graph we can order and assign an orientation to the edges and therefore define the corresponding incidence matrix $B : \R^E \rightarrow \R^V$ by
\begin{align*}
B_{ie} =
\begin{cases}
1 & \mbox{if the vertex $i$ is the head of the edge $e$,}
\\
-1 & \mbox{if the vertex $i$ is the tail of the edge $e$,}
\\
0 & \mbox{otherwise.}
\end{cases}
\end{align*}
The integral vectors in the kernel of $B$ form an additive group called the cycle space since each such vector corresponds to a cycle in the graph $G$. Let $v_1,\dots,v_c$ denote a basis for the cycle space. It is well known that $c = |E| - |V| + 1$. It is not hard to see that any two such basis can be obtained from the other by an element of $SL_c(\Z)$, the special linear group of degree $c$ over $\Z$.

If $e$ is an edge with head $i$ and tail $j$ we define the angle difference $\theta_e = \theta_i - \theta_j$. When we consider fixed points of $\eqref{model}$ and $\eqref{simplified model}$ we will restrict these angle differences so that $\theta_e \in I_e + 2\pi \Z$ for some choice of sets $I_e$. Throughout we will assume that each set is such that the restriction of the sine function to $I_e$, $\sin_{I_e}$, is injective. Given a collection of such sets $\mathcal{I} = \{ I_e \}_{e \in E}$ we define the function $\sin_\I$ to be the $|E|$ vector valued sine function restricted to the set $I_{e_1} \times \dots \times I_{e_{|E|}}$. Finally if $\gamma_e$ is the weight of the edge $e$ we let $D_\gamma$ be the diagonal matrix with diagonal entries $\gamma_{e_1}, \dots, \gamma_{e_{|E|}}$.

We can now define two functions on $\R^c$ by
\begin{align*}
L(\alpha) = D_\gamma^{-1} B^{-1} {\bf \omega} + \sum_{i=1}^c \alpha_i D_\gamma^{-1} v_i \quad \text{and} \quad W(\alpha) = \frac{1}{2\pi} (\langle v_1, \sin_\mathcal{I}^{-1} L(\alpha) \rangle, \dots,  \langle v_c, \sin_\mathcal{I}^{-1} L(\alpha) \rangle)
\end{align*}
where $B^{-1}$ is the pseudoinverse of $B$. We further let $A$ denote the subset of $\R^c$ on which $\sin_\mathcal{I}^{-1} L(\alpha)$ is well-defined. In fact $A$ is a polytope, and we will study it further in Theorem $\ref{polytope}$. One will notice that both $L$ and hence $W$ are injective on $A$. We can now state our first main result.


\begin{theorem} \label{equivalence}
There is a one-to-one correspondence between the fixed points of $\eqref{model}$ with $\theta_e \in I_e + 2\pi \Z$ and the lattice points in $W(A)$.
\end{theorem}



\begin{example}
Let $G$ be the expanded diamond graph with unit edge weights shown in Figure 1 below. Also let $I_e = [-\pi/2,\pi/2]$ and
\begin{gather*}
v_1 = (0, -1, 1, 0, 0, 0, 0, 0, 0, 0, 1, -1, 1, 1, 1, 1, 1, 1, 1, 1), \\
v_2 = (-1, 0, 1, -1, -1, -1, -1, -1, -1, -1, 0, -1, 0, 0, 0, 0, 0, 0, 1, 1), \\
K = (1, -1, 0, 0, 0, 0, 0, 0, 0, 0, 1, 0, 0, 0, 0, 0, 0, 0, 0, 0).
\end{gather*}
We will see the meaning of $K$ momentarily. From Figure 1 we see that $W(A)$ contains the lattice point $(2,-1)$ and therefore by Theorem $\ref{equivalence}$ there is a corresponding fixed point of $\eqref{model}$. For simplicity we will work with $\eqref{simplified model}$. We will compute this fixed point which will also allow us to sketch the proof of Theorem $\ref{equivalence}$ which will be given in complete detail in the appendix.

To begin we rewrite the fixed point equation of $\eqref{simplified model}$ as $B \sin B^\top \theta = 0$. Since $v_1$ and $v_2$ span the kernel of $B$ we see that $\sin B^\top \theta = L(\alpha,\beta)$ for some $(\alpha,\beta) \in A$ and therefore that $B^\top \theta = \sin_\I^{-1} L(\alpha,\beta) - 2\pi K$ for some $K \in \Z^{20}$. If we let $B^{-\top}$ denote the pseudo inverse of $B^\top$ then we obtain the formula $\theta = B^{-\top}(\sin_\I^{-1} L(\alpha,\beta) - 2\pi K)$. Thus it remains to compute $(\alpha,\beta)$ and $K$. Since $v_1$ and $v_2$ are orthogonal to the image of $B^\top$ we see that $W(\alpha,\beta) = (\langle v_1,K \rangle,\langle v_2,K \rangle) = (1,0)$. This works with our choice for $K$, and we numerically solve to find that $(\alpha,\beta) = (0.994148 \dots, -0.779356 \dots)$. This finally results in
\begin{gather*}
\theta = (0.724472, 1.61811, 2.51175, 3.40538, 4.29902, 5.19266, 6.0863, 0.696749, 1.59039, 3.05294, \\ 4.5155, 5.97806, 1.15743, 2.61999, 4.08254, 5.5451, 0.94095, 1.15743, 1.37391)
\end{gather*}
which is displayed in Figure 1 below.
\begin{figure}[H]
\captionsetup{font=scriptsize}
\centering
\begin{subfigure}{.5\textwidth}
  \centering
  \includegraphics[width=.65\linewidth]{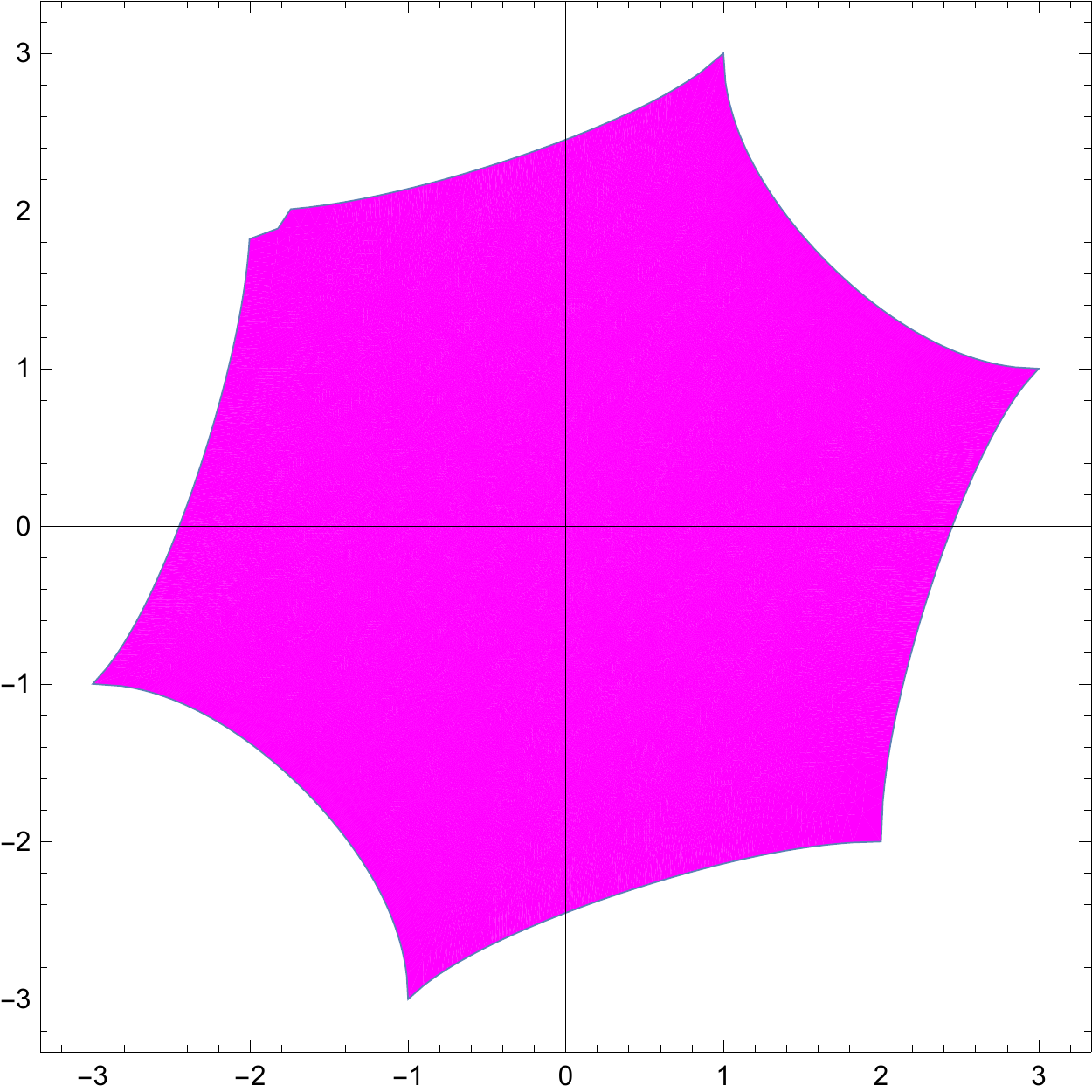}
\end{subfigure}%
\begin{subfigure}{.5\textwidth}
  \centering
  \includegraphics[width=1\linewidth]{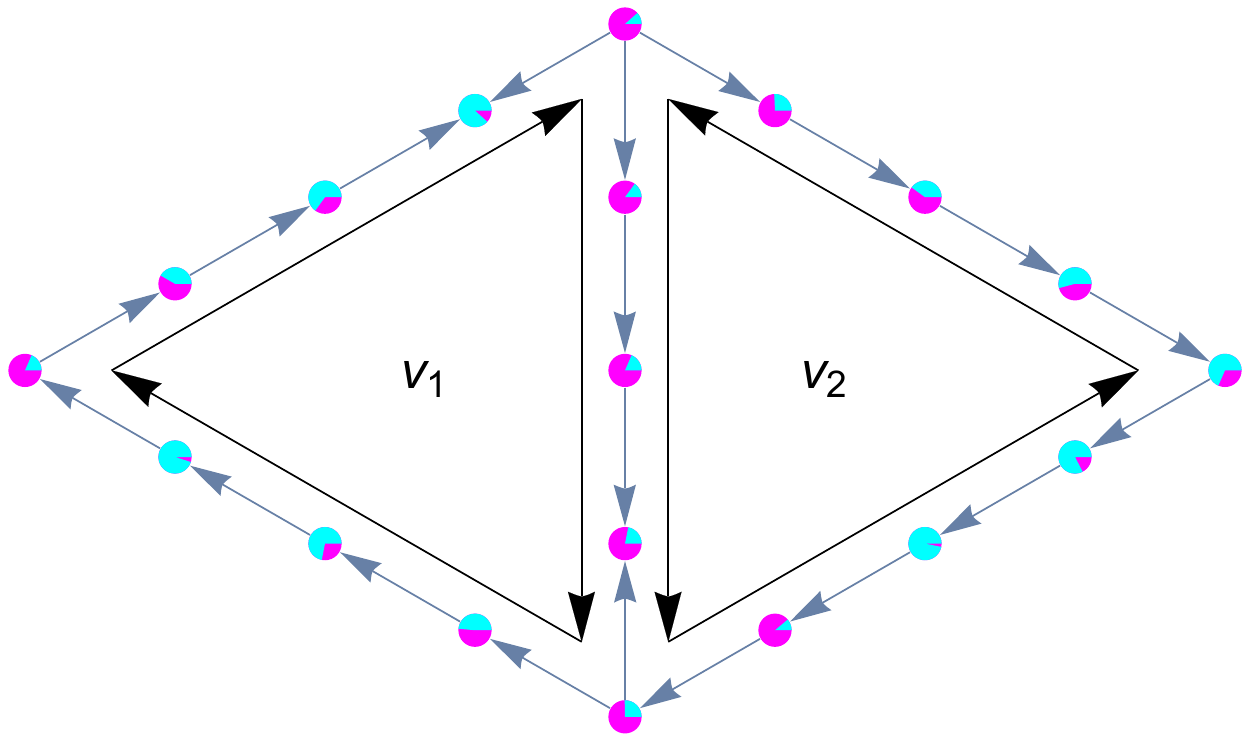}
\end{subfigure}
\caption{The plot of $W(A)$ contains the lattice point $(2,-1)$. The angles of the corresponding fixed point are indicated by the cyan sectors at each vertex. Notice that the angles have a net winding of one and zero as we move around the cycles $v_1$ and $v_2$ respectively. This is in agreement with our lattice point $(2,-1)$.}
\end{figure}
Furthermore since we chose $I_e = [-\pi/2,\pi/2]$ and positive edge weights we know that our fixed point must be stable which we illustrate in Figure 2 below. Of course this holds in general.
\begin{figure}[H]
\captionsetup{font=scriptsize}
\centering
\includegraphics[width=.5\linewidth]{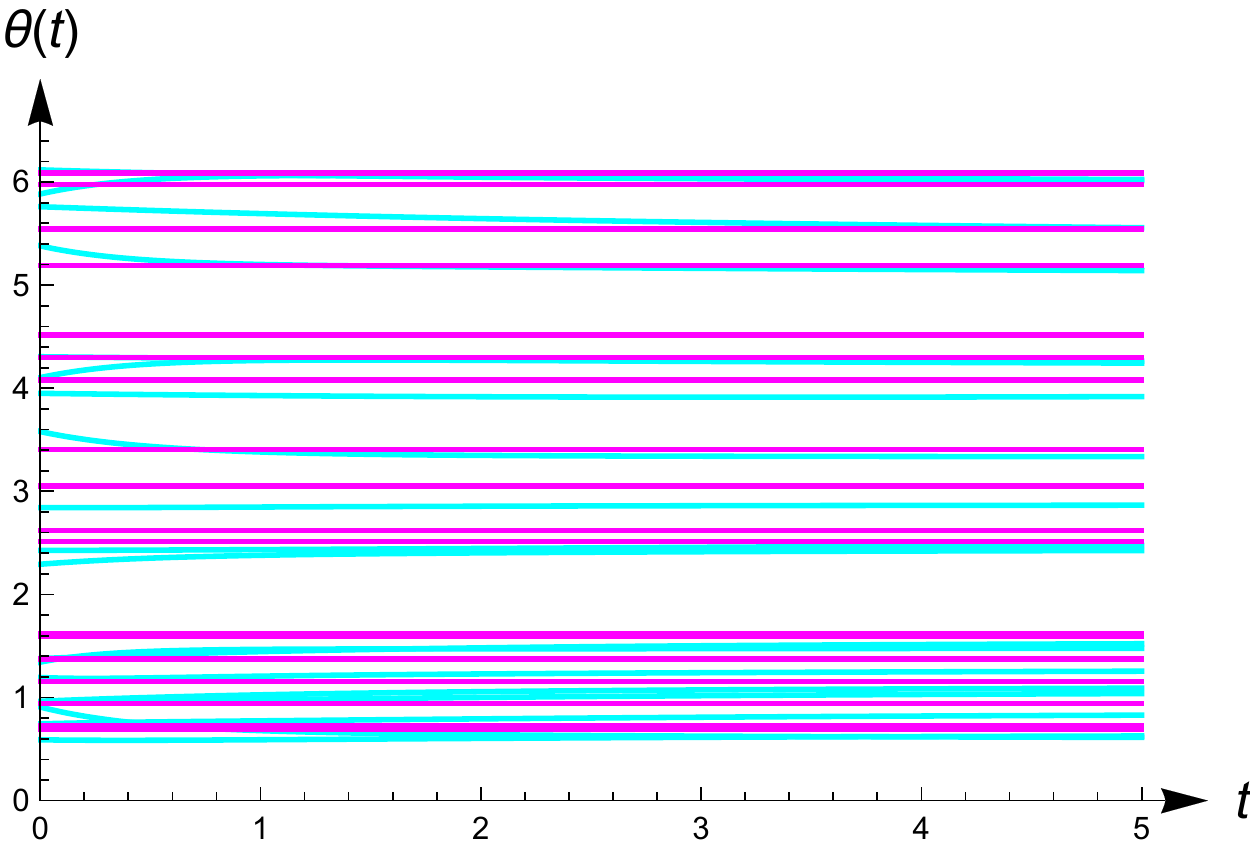}
\caption{The plot of the evolution of the Kuramoto model with initial data which is a perturbation of our fixed point by a random vector with components bounded by $0.25$.}
\end{figure}
It is worth noting that since $W$ is always injective on $A$ we can in fact associate a unique $\alpha \in A$ to any steady state $\theta$ of $\eqref{model}$ via the equation $L(\alpha) = \sin(B^\top \theta)$. We will use this association in Theorem 5 in Section 3.
\end{example}


One will notice that the fixed points of $\eqref{model}$ with $\theta_e \in I_e +2\pi \Z$ are completely independent of our choice of cycle basis whereas $W(A)$ explicitly depends on it. We will briefly explain why the number of lattice points in $W(A)$ is independent of the choice of cycle basis directly from the definition of $W$ and $A$. To do this let $G$ be a graph with two cycle basis $v_1,\dots,v_c$ and $v_1',\dots,v_c'$. Let $A$ and $A'$, $L$ and $L'$, and $W$ and $W'$ be the corresponding polytopes and functions respectively. As mentioned earlier these two cycle basis are related by an element of $SL_c(\Z)$. In particular there exists an $M \in SL_c(\Z)$ such that $(v_1',\dots,v_c') = (v_1,\dots,v_c)M$. From this we see that $L'(\alpha) = (v_1',\dots,v_c') \alpha = (v_1,\dots,v_c)M\alpha = L(M\alpha)$ and therefore that $L'(M^{-1} A) = L(A) \subseteq \im(\sin_\I)$ which implies that $M^{-1}A \subseteq A'$. Reasoning again in this way we are ultimately lead to the identities $A = MA'$ and $A' = M^{-1}A$. Just like $L$ and $L'$ are related by $M$ we see that $W'(\alpha) = L'(\alpha)(v_1',\dots,v_c') = L(M\alpha)(v_1,\dots,v_c)M = W(M\alpha)M$. From this we are finally lead to the identity $W'(A') = W(A)M$. Since $M$ is an element of $SL_c(\Z)$ we see that it is a bijection between the lattice points in $W(A)$ and $W'(A')$. Thus again we see that our result is independent of the choice of cycle basis.

Along the same lines we see that the polytope $A$ is not well-defined until we specify a cycle basis. However, as we have already seen, any two polytopes are related by an element of the special linear group and therefore there ought to be properties of $A$ which are independent of the particular choice of cycle basis. We give an example of such a property in Theorem $\ref{polytope}$, but first we need to discuss the smoothing of a graph.

If a graph has a 2-valent vertex it is possible to obtain a new graph by removing this vertex and replacing the two edges with a single edge. When this is done we say that the vertex has been smoothed. We can repeat this process and in so doing we obtain a new graph which potentially contains self loops and edges with multiplicity greater than two. Our only restriction is that if our graph is only a single self loop we do not remove that 2-valent vertex.


\begin{theorem} \label{polytope}
Suppose that $\sin(I_e) = [-1,1]$. Then the number of faces of the polytope $A$ is equal to twice the number of edges of the graph obtained from $G$ by smoothing all 2-valent vertices.
\end{theorem}


Notice that if we smooth out all 2-valent vertices that we don't change the number of distinct components of $L(\alpha)$. Thus the theorem is simply saying that each of the hyperplanes $L(\alpha)_e = \pm 1$ results in a face of $A$.


\begin{example}
Consider the left most graph shown in Figure 1 below. After smoothing the 2-valent vertices 3, 4, 5, and 6 we obtain the right most graph. Thus according Theorem $\ref{polytope}$, $A$ has eight faces which we can easily confirm see from Figure 3.
\begin{figure}[H]
\captionsetup{font=scriptsize}
\center
\includegraphics{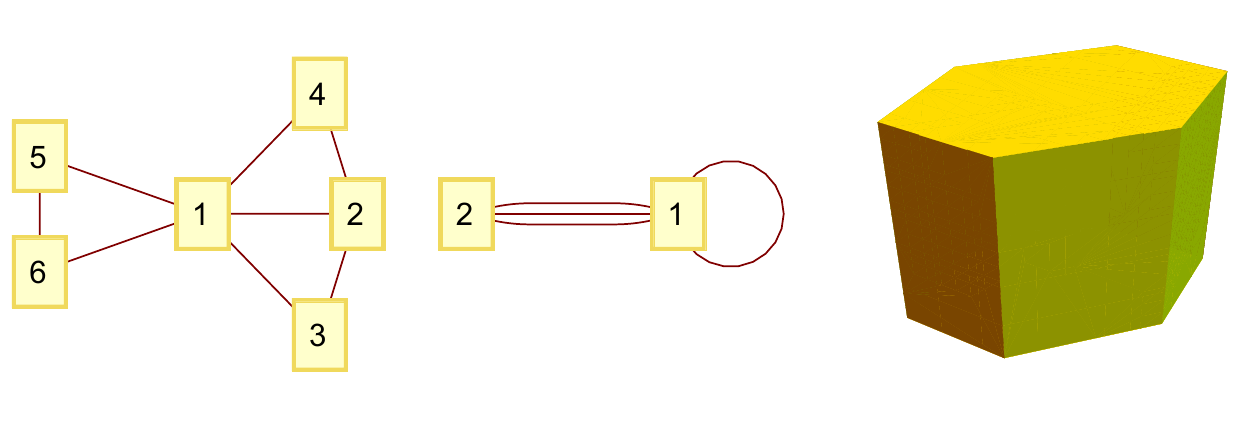}
\caption{$G$ before and after smoothing all 2-valent vertices and $A$}
\end{figure}
Furthermore, with a suitable choice of orientation and numbering of the edges we can choose the cycle basis $v_1 = (1,1,1,0,0,0,0,0)$, $v_2 = (0,0,1,1,1,0,0,0)$, and $v_3 = (0,0,0,0,0,1,1,1)$ for the left most graph. Thus $L(\alpha,\beta,\gamma) = (\alpha,\alpha,\alpha+\beta,\beta,\beta,\gamma,\gamma,\gamma)$ and we see that we get the eight hyperplanes $\alpha = \pm1$, $\beta = \pm 1$, $\gamma = \pm1$, and $\alpha + \beta = \pm 1$.
\end{example}


Now we consider an application of Theorem $\ref{equivalence}$. Let $G$ be a graph and let $G_M$ be a sequence of graphs such that each edge $e$ in $G$ is replaced by a path of edges of length $r_e M + o(M)$ for some $r_e > 0$. For each new edge $e'$ set $I_{e'} = I_e$ where $e$ is the edge that was replaced by the path of edges containing $e'$. Essentially we are doing the exact opposite of smoothing 2-valent vertices which is referred to as subdivision. An example of a few graphs in such a sequence is shown in Figure 4 below.
\begin{figure}[H]
\captionsetup{font=scriptsize}
\center
\includegraphics{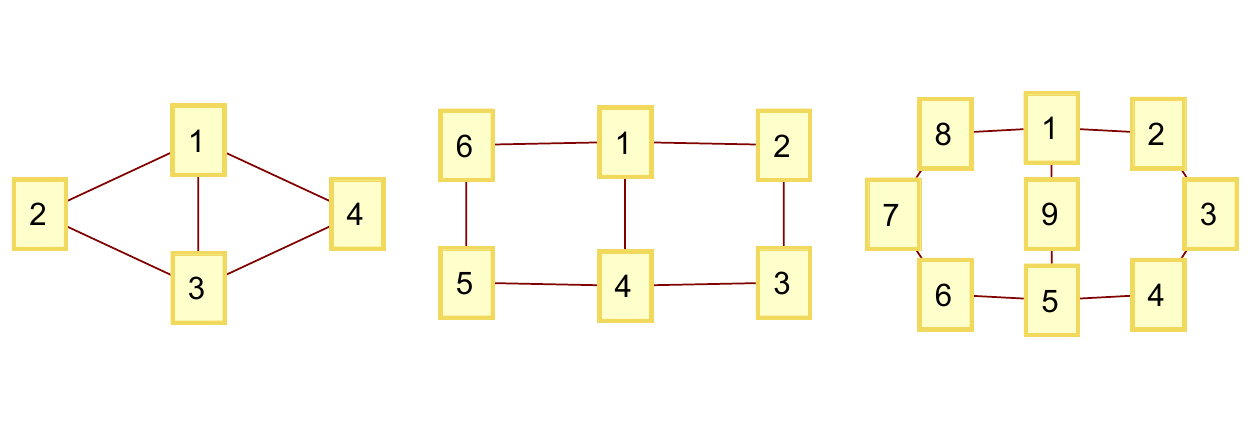}
\caption{$G_1$, $G_2$, and $G_3$}
\end{figure}


\begin{theorem} \label{asymptotics}
Let $G_M$ be a sequence of graphs as described above and let $D_r$ be the diagonal matrix with entries $r_e$. Then
\begin{align*}
\lim_{M \rightarrow \infty} \frac{\#(W_M(A_M) \cap \Z^c)}{M^c} = |W_r(A_0)|
\end{align*}
where
\begin{align*}
W_r(\alpha) = \frac{1}{2\pi} (\langle v_1, D_r \sin_\mathcal{I}^{-1} L_0(\alpha) \rangle, \dots, \langle v_c, D_r \sin_\mathcal{I}^{-1} L_0(\alpha) \rangle).
\end{align*}
In particular, if $r_e = 1$, then the limit equals $|W(A)|$.
\end{theorem}


Essentially since $W_M(A_M)$ expands proportionally to $M$ we rescale and instead count the number of elements of $(\frac{1}{M} \Z)^c$ in $W_M(A_M)/M$. Since the sequence of sets $W_M(A_M)/M$ approaches a "limiting set" we can estimate this number by $M^c$ times a Riemann sum of the indicator function for this "limiting set". In the case special case that $r_e = 1$ the "limiting set" is in fact $W(A)$.

Theorem $\ref{asymptotics}$ indicates that the main contribution to the number of fixed points of $\eqref{model}$ when $r_e = 1$ comes from $|W(A)|$. Therefore if we seek a large number of fixed points we ought to try to maximize $|W(A)|$. The only way we can do this without changing the graph is by changing the intervals $I_e$. This motivates the following theorem.


\begin{theorem} \label{angle differences}
The maximum value of $|W(A)|$ is achieved if we choose
\begin{align*}
I_e =
\begin{cases}
[-\pi/2,\pi/2] & \mbox{if $\gamma_e > 0$,}
\\
[\pi/2,3\pi/2] & \mbox{if $\gamma_e < 0$,}
\end{cases}
\quad \text{or} \quad
I_e =
\begin{cases}
[\pi/2,3\pi/2] & \mbox{if $\gamma_e > 0$,}
\\
[-\pi/2,\pi/2] & \mbox{if $\gamma_e < 0$.}
\end{cases}
\end{align*}
In particular, if all of the edge weights are positive, then the maximum is achieved if all of the angle differences are less than or equal to $\pi/2$ or all greater than or equal to $\pi/2$.
\end{theorem}


This theorem follows from the identity $|W(A)| = \int_A |\det(W')|$ and the fact that the given interval choices prevent cancellation among the terms in the determinant. In applications such as power systems the special case $I_e = [-\pi/2,\pi/2]$ is particularly useful. Furthermore if all of the edge weights are positive and all of the angle differences are less than $\pi/2$, the steady state is automatically stable. Thus our estimates can be used to obtain a lower bound for the number of stable steady states. More generally, if the intervals are chosen as in the left most case of Theorem $\ref{angle differences}$, then the steady state is automatically stable.



\section{Graphs with one cycle}

In this section we study single cycle graphs with the additional assumption that $\sin(I_e) = (-1,1)$. If $G$ has one cycle, then it is a ring. (Recall that we are assuming that $G$ doesn't have any edges which are contained in every spanning tree of $G$.) Thus we can orient all edges in one direction so that a cycle basis is given by the single vector $v = (1,\dots,1)$. Furthermore we note that $\alpha$ is a scalar and no longer a vector. Thus $L(\alpha) = (\alpha/\gamma_{e_1},\dots,\alpha/\gamma_{e_{|E|}})$ and $W(\alpha) = \sum_{e \in E} \sin_{I_e}^{-1} (\alpha/\gamma_e)$. In addition $A = (-\min_{e \in E} |\gamma_e|, \min_{e \in E} |\gamma_e|)$.


\begin{theorem} \label{extreme winding}
Define $I_+ = (0,\pi/2) \cup (\pi,3\pi/2)$ and $I_- = (\pi/2,\pi) \cup (3\pi/2,2\pi)$. Given a fixed point $\theta$ of $\eqref{simplified model}$ let $\alpha$ be such that $L(\alpha) = \sin(B^\top \theta)$. If one of the inequalities
\begin{align*}
W(\alpha) > \max_{e \in E} \biggr\{ \sin_{I_-}^{-1} \left( \frac{\alpha}{\gamma_e} \right) + \sum_{e' \ne e} \sin_{I_+}^{-1} \left( \frac{\alpha}{\gamma_{e'}} \right) \biggr\}, \quad \alpha \ge 0,
\\
W(\alpha) < \min_{e \in E} \biggr\{ \sin_{I_+}^{-1} \left( \frac{\alpha}{\gamma_e} \right) + \sum_{e' \ne e} \sin_{I_-}^{-1} \left( \frac{\alpha}{\gamma_{e'}} \right) \biggr\}, \quad \alpha \le 0,
\end{align*}
holds, then $\theta$ is unstable.
\end{theorem}



\begin{proof}
By Theorem 2.9 of \cite{MR3513871} $\theta$ is unstable if there are at least two edges with $\gamma_e \cos \theta_e < 0$. As we shall see in the proof of Theorem $\ref{equivalence}$ we have that $\theta_e$ is in fact equal to $\sin_{I_e}^{-1} (\alpha/\gamma_e)$ modulo an integer multiple of $2\pi$. If $\alpha \ge 0$ one can easily check that $\sin_{I_+}^{-1}(\alpha/\gamma_e) < \sin_{I_-}^{-1}(\alpha/\gamma_e)$, $\gamma_e \cos(\sin_{I_+}^{-1}(\alpha/\gamma_e)) > 0$, and $\gamma_e \cos(\sin_{I_-}^{-1}(\alpha/\gamma_e)) < 0$. Thus the first inequality can only hold if $\gamma_e \cos \theta_e < 0$ for at least two edges which results in instability. The second inequality follows by the same reasoning since if $\alpha \le 0$ we have that $\sin_{I_+}^{-1}(\alpha/\gamma_e) < \sin_{I_-}^{-1}(\alpha/\gamma_e)$, $\gamma_e \cos(\sin_{I_+}^{-1}(\alpha/\gamma_e)) < 0$, and $\gamma_e \cos(\sin_{I_-}^{-1}(\alpha/\gamma_e)) > 0$.
\end{proof}



\begin{example}
Consider the cycle graph with unit edge weights and the $q$-twisted state $\theta_i = 2\pi iq/N$. If $q \in (N/4,N/2)$, then $\theta_e = 2\pi q/N \in (\pi/2,\pi) \subseteq I_-$. As we will see in the proof of Theorem $\ref{equivalence}$ we have that $B^\top \theta = \sin_\I^{-1} L(\alpha) - 2\pi K$ for some $K \in \Z^E$ hence $\alpha = \sin(2\pi q/N) \ge 0$. Therefore since $W(\alpha) = N \sin_{I_-}^{-1} (\alpha) > \sin_{I_-}^{-1} (\alpha) + (N -1) \sin_{I_+}^{-1} (\alpha)$ we see that this twisted state is unstable by Theorem $\ref{extreme winding}$. Similarly, if $q \in (N/2,3N/4)$, then $\theta_e \in (\pi,3\pi/2) \subseteq I_+$ and $\alpha = \sin(2\pi q/N) \le 0$. We again conclude that the twisted state is unstable since $W(\alpha) = N \sin_{I_+}^{-1} (\alpha) < \sin_{I_+}^{-1} (\alpha) + (N-1) \sin_{I_-}^{-1} (\alpha)$. It is worth noting that $\cos (2\pi q/N) < 0$ in both cases. One can in fact show directly that such a twisted state is unstable if and only if $\cos (2\pi q/N) < 0$.
\end{example}




\section{Graphs with two cycles}

In this section we will study two cycle graphs with the additional assumption that $\sin(I_e) = [-1,1]$. Due to the translation invariance of fixed points of $\eqref{simplified model}$ we observe that if we glue two graphs together at a single vertex that the number of fixed points of this graph is simply the product of the number of fixed points of the two subgraphs. Therefore we assume our graph is composed of two cycles which intersect nontrivially. Let $v_1$ and $v_2$ be a cycle basis. We can assume that all of the components of $v_1$ are zero or one and that the components of $v_2$ are zero or one outside of the intersection of the two cycles and negative one on this intersection. With this choice of cycle basis our assumption about our graph implies that $L(\alpha,\beta)_e \in \{\alpha,\beta,\alpha-\beta \}$ with $\alpha -\beta$ being realized for at least one edge. Thus we see that $A = \{ (\alpha,\beta) : |\alpha| \le 1, |\beta | \le 1, |\alpha - \beta | \le 1 \}$ which is shown in Figure 5 above.
\begin{figure}
\captionsetup{font=scriptsize}
\center
\includegraphics[width=150px,height=150px]{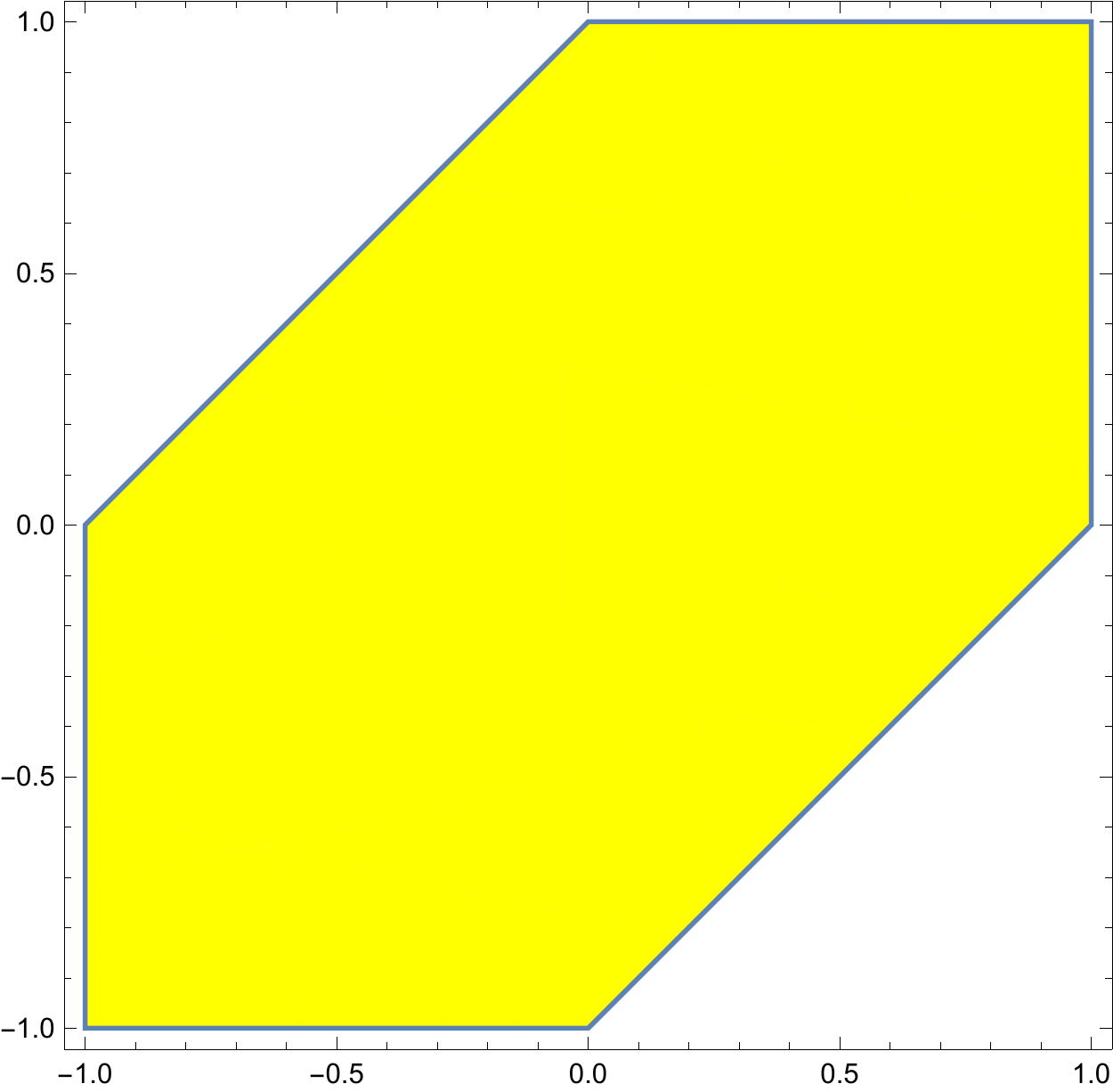}
\caption{$A$ for the generic two cycle case}
\end{figure}
It turns out that the set $\{\alpha,\beta,\alpha-\beta \}$ has a lot of symmetry which allows us to prove the following theorem.


\begin{theorem} \label{cycle intersection matrix}
Suppose that all of the edge weights are equal to a common value $\gamma$ and that $I_e = [-\pi/2,\pi/2]$ for all edges or $I_e = [\pi/2,3\pi/2]$ for all edges. Then
\begin{align*}
|W(A)| = \frac{N |\spt(G)|}{|\gamma|^2} \left(\frac{\pi^2}{2} + 2 \int_0^1 \frac{\sin^{-1} (1-\beta)}{\sqrt{1-\beta^2}} d\beta \right).
\end{align*}
\end{theorem}



\begin{proof}
Since $W$ is injective on $A$ we have that $|W(A)| = \int_A |\det(W'(\alpha))| d\alpha$. By Theorem 2.8 of \cite{MR3513871} we have the identity
\begin{align*}
|\det(W'(\alpha))| = \frac{N}{|\gamma|^c} \sum_{T \in \spt(G)} \prod_{e \notin E_T} \frac{1}{\sqrt{1-L(\alpha)_e^2}}
\end{align*}
for any number of cycles. In the notation of \cite{MR3513871}, $W'(\alpha)$ is a cycle intersection matrix for the graph with edge weights $\gamma_e \sqrt{1-L(\alpha)_e^2}$. We will now argue that in the case of two cycles the integrals
\begin{align*}
I_T = \int_A \prod_{e \notin E_T} \frac{1}{\sqrt{1-L(\alpha)_e^2}} d\alpha
\end{align*}
are independent of the spanning tree. To see this notice that on the complement of any spanning tree $L(\alpha,\beta)$ takes exactly two distinct values from the set $\{\alpha,\beta,\alpha-\beta \}$. However, there exists a linear transformation $P(\alpha,\beta) = (\beta,-\alpha+\beta)$ which along with its square maps any two such values to any other two with at most a sign change. Moreover the determinant of $P$ has magnitude one. Thus we can use $P$ and its square to switch between spanning trees without changing the value of the integral. Thus to compute this integral we can choose a spanning tree whose complement contains an edge from $v_1$ not contained in $v_2$ and an edge from $v_2$ not contained in $v_1$. This results in the integral
\begin{align*}
I = \int_A \frac{d\alpha d\beta}{\sqrt{(1-\alpha^2)(1-\beta^2)}} = \frac{\pi^2}{2} + 2 \int_0^1 \frac{\sin^{-1} (1-\beta)}{\sqrt{1-\beta^2}} d\beta = 0.64368 \dots
\end{align*}
\end{proof}


Unfortunately the independence of the integral on the spanning tree does not seem to hold in general for graphs with more than two cycles. For example consider the graph in Figure 6 below along with two of its spanning trees.
\begin{figure}[H]
\captionsetup{font=scriptsize}
\center
\includegraphics[width=400px,height=150px]{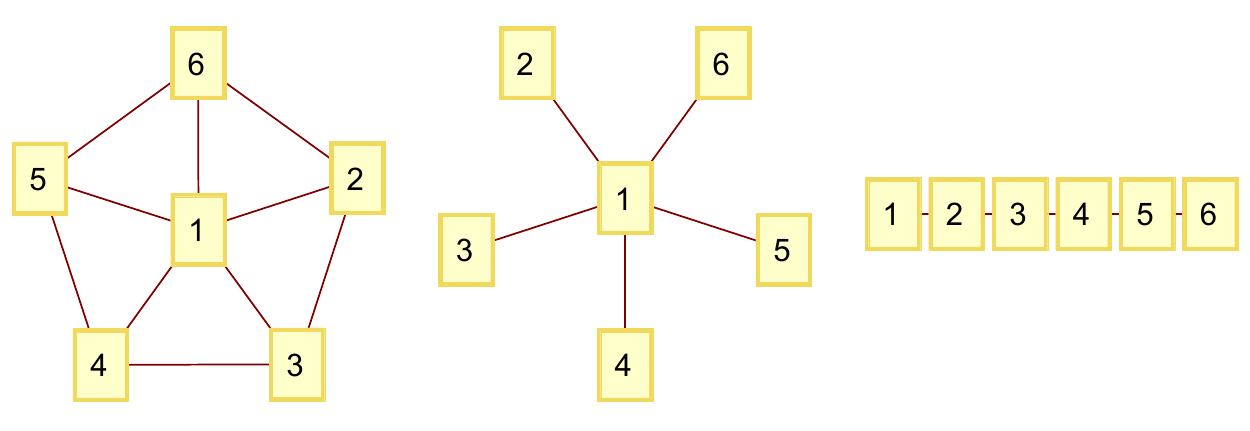}
\caption{$G$, $T_1$, and $T_2$}
\end{figure}
Then $I_{T_1} \approx 36 \ne 44 \approx I_{T_2}$. These numerical values were obtained with Mathematica by means of a Monte Carlo method with precision goal two. Of course we do have the inequalities
\begin{align*}
\frac{N |\spt(G)|}{|\gamma|^c} \min_{T \in \spt(G)} I_T \le |W(A)| \le \frac{N |\spt(G)|}{|\gamma|^c} \max_{T \in \spt(G)} I_T
\end{align*}
which hold in general.



\section{Acknowledgements}

The author gratefully acknowledges support under NSF grant DMS1615418.



\section{Appendix}

In this section we provide proofs of our theorems. We start with a proof of Theorem $\ref{equivalence}$ which requires the following lemma.


\begin{lemma} \label{lattice}
Let $v_1,\dots,v_c$ be a cycle basis. Then
\begin{align*}
\spn_\mathbb{Z}
\begin{pmatrix}
v_1
\\
\vdots
\\
v_c
\end{pmatrix}
= \mathbb{Z}^c
\end{align*}
where the span consists of all integral linear combinations of the columns of the matrix.
\end{lemma}



\begin{proof}
We start by making two simple observations. First, notice that the lemma is independent of the particular choice of cycle basis since any two cycle basis are related by an element of $SL_c(\mathbb{Z})$. Second, the lemma is trivial for graphs with a single cycle. From these two observations we see that the lemma will follow by induction once we show that a graph with a cycle basis $v_1, \dots, v_c$ satisfying the lemma implies that a graph with one more edge and one new cycle basis vector transversing this new edge exactly once satisfies the lemma.

In other words the new cycle basis $v_1',\dots, v_{c+1}'$ satisfies $v_i' = (v_i,0)$ for $i \in \{1,\dots,c\}$ and $v_{c+1}' = (v_{c+1}, \pm1)$ for some integral $v_{c+1}$. Thus
\begin{align*}
\spn_\mathbb{Z}
\begin{pmatrix}
v_1'
\\
\vdots
\\
v_{c+1}'
\end{pmatrix}
=
\spn_\mathbb{Z}
\begin{pmatrix}
v_1 & 0
\\
\vdots
\\
v_c & 0
\\
v_{c+1} & \pm 1
\end{pmatrix}
= \mathbb{Z}^{c+1}
\end{align*}
by the inductive hypothesis which completes the proof.
\end{proof}



\begin{proof}[Proof of Theorem $\ref{equivalence}$]
We start by noting that $\eqref{model}$ is equivalent to the equation $\omega = BD \sin B^\top \theta$. Recalling that a cycle basis is also a basis for $\ker(B)$ we easily find that $\theta$ is a solution of $\eqref{model}$ if and only if $\sin B^\top \theta = L(\alpha)$ for some $\alpha \in \R^c$. Since $\theta_e \in I_e + 2\pi \mathbb{Z}$ we see that in fact $\alpha \in A$ and also that $B^\top \theta = \sin_\mathcal{I}^{-1} L(\alpha) - 2\pi K$ for some $K \in \mathbb{Z}^E$. Note that $\theta$ is uniquely determined up to translation by $\alpha$. Of course this equation holds if and only if $\frac{1}{2\pi} \langle v_i, \sin_\mathcal{I}^{-1} L(\alpha) \rangle = \langle v_i, K \rangle$. In other words $\theta$ is a solution of $\eqref{model}$ if and only if there exists an $\alpha \in A$ such that
\begin{align}
W(\alpha) \in \spn_\mathbb{Z}
\begin{pmatrix}
v_1
\\
\vdots
\\
v_c
\end{pmatrix}
= \mathbb{Z}^c
\end{align}
where the equality holds by Lemma $\ref{lattice}$.
\end{proof}


Now we prove Theorem $\ref{polytope}$.


\begin{proof}[Proof of Theorem $\ref{polytope}$]
We do this by induction. If $G$ has only one cycle the result is obvious. Thus let $G$ be a graph with at least two cycles. If we choose the cycle basis for $G$ to be a  fundamental cycle basis, then there exists two edges $e'$ and $e''$ which are only transversed by exactly one member of the cycle basis. This allows us to define two subgraphs $G'$ and $G''$ which are obtained from $G$ by deleting the edges $e'$ and $e''$ respectively. We will show by applying the inductive hypothesis to $G'$ that the intersection of $A$ with the hyperplane $L(\alpha)_e = h_e \in \{-1,1\}$ has non-empty interior for all edges $e \ne e'$. By applying the same reasoning to the subgraph $G''$ we will complete the inductive step. To do this let $A$ and $A'$ be the corresponding polytopes and $L'(\alpha)$ and $L(\alpha,\beta)$ be the corresponding linear functions defined on these polytopes. Furthermore let $v$ denote the cycle basis vector which transverses $e'$ only once and completes a cycle basis for $G$ when added to the cycle basis for $G'$. Note that we can assume that $v_e \in \{-1,0,1\}$ and therefore it suffices to consider the two cases $v_e = 0$ and $v_e \in \{-1,1\}$.

We first consider the case when $v_e = 0$ which implies that $L(\alpha,\beta)_e = L'(\alpha)_e$. Let $\epsilon > 0$ be a small parameter and define the set $C_\epsilon= \{ \alpha \in A' : L'(\alpha)_e = h_e, \min_{e^* \ne e} |L'(\alpha)_{e^*} \pm 1| \ge \epsilon \}$. By the inductive hypothesis we know that this set has non-empty interior for all sufficiently small $\epsilon$. Now the set $C_\epsilon \times [-\epsilon,\epsilon]$ lies in the intersection of $A$ and the hyperplane $L(\alpha,\beta)_e = h_e$ and has non-empty interior.

Next we consider the case when $v_e \in  \{-1,1\}$ which implies that $L(\alpha,\beta)_e = L'(\alpha)_e + v_e \beta$. Define $\beta(\alpha) = v_e(h_e - L'(\alpha)_e)$, then the graph $(\alpha,\beta(\alpha))$ lies on the hyperplane $L(\alpha,\beta)_e = h_e$. Define the set $C = \{ \alpha \in A' : |L'(\alpha)_e - h_e| \le \min_{e^* \neq e} |L'(\alpha)_{e^*} \pm 1| \}$. Note that this set has non-empty interior by the inductive hypothesis. If $v_{e^*} = 0$, then $L(\alpha,\beta(\alpha))_{e^*} \in [-1,1]$ trivially. Thus suppose that $v_{e^*} \in \{-1,1\}$. Then on $C$ we have that
\begin{align*}
1-h_eL'(\alpha)_e \le 1-v_e h_e v_{e^*} L'(\alpha)_{e^*}
\end{align*}
which implies that
\begin{align*}
L(\alpha,\beta(\alpha))_{e^*} = v_{e^*}(v_e h_e -v_e L'(\alpha)_e + v_{e^*} L'(\alpha)_{e^*}) \in [-1,1]
\end{align*}
by considering the two cases $v_e h_e \in \{-1,1\}$. Therefore the graph $(\alpha,\beta(\alpha))$ on $C$ lies in the intersection of $A$ and the hyperplane $L(\alpha,\beta)_e = h_e$ and has non-empty interior.
\end{proof}


We now prove Theorem $\ref{asymptotics}$.


\begin{proof}
First notice that $A_M = A$. By construction we have that $W_M(\alpha) = MW_r(\alpha) + E_M(\alpha)$ where $\sup_{\alpha \in A} |E_M(\alpha)| = o(M)$. Therefore
\begin{align*}
\#(W_M(A) \cap \Z^c) = \#(\{ W_r(\alpha) + E_M(\alpha)/M : \alpha \in A\} \cap (\frac{1}{M}\Z)^c )
\end{align*}
If we let $R_M$ denote a Riemann sum for the indicator function of the set $\{ W_r(\alpha) + E_M(\alpha)/M : \alpha \in A\}$ with mesh $(\frac{1}{M}\Z)^c$, then the above quantity equals $R_M M^c$. Thus it suffices to show that $R_M$ approaches $|W_r(A)|$ as $M$ tends to infinity. Notice that the boundary $\partial W_r(A)$ is compact and that
\begin{align*}
|R_M - |W_r(A)|| \le |\{ \beta : \dist(\beta,\partial W_r(A)) \le \sup_{\alpha \in A} |E_M(\alpha)|/M + \sqrt{c}/M \}|.
\end{align*}
where $\dist(\beta,W_r(A)) = \inf_{\alpha \in A} |\beta - W_r(\alpha)|$. Thus we obtain our desired result since the quantity bounding the distance between $\beta$ and $\partial W_r(A)$ tends to zero.

\end{proof}


Finally we prove our last theorem which is Theorem $\ref{angle differences}$.


\begin{proof}

As noted before we have that $|W(A)| = \int_A |\det(W'(\alpha)| d\alpha$ since $W$ is injective on $A$. If we define $\epsilon_e(\alpha) \in \{-1,1\}$ such that
\begin{align*}
\frac{d}{d\alpha} \sin_{I_e}^{-1} L(\alpha)_e = \frac{\epsilon_e(\alpha)}{\sqrt{1-L(\alpha)_e^2}} \quad \text{then} \quad (W')_{ij} = \sum_{e \in E} \frac{(v_i)_e (v_j)_e}{\epsilon_e(\alpha) \gamma_e \sqrt{1-L(\alpha)_e^2}}.
\end{align*}
Again $W'(\alpha)$ is a cycle intersection matrix in the notation of \cite{MR3513871} with edge weights $\epsilon_e(\alpha) \gamma_e \sqrt{1-L(\alpha)_e^2}$. Therefore by Theorem 2.8 of \cite{MR3513871} we find that
\begin{align*}
\det(W'(\alpha)) = N \sum_{T \in \spt(G)} \prod_{e \notin E_T} \frac{1}{\epsilon_e(\alpha) \gamma_e \sqrt{1-L(\alpha)_e^2}}.
\end{align*}
Thus the absolute value of the determinant is maximized when $\prod_{e \notin E_T} \epsilon_e(\alpha) \gamma_e$ has the same sign for all spanning trees $T$. It is not difficult to convince oneself that this implies that $\epsilon_e(\alpha) \gamma_e$ must have the same sign for any edge $e$ which is in the complement of some spanning tree $T$. By assumption this is all of $E$ and therefore for each $\alpha$ we see that $\epsilon_e(\alpha) \gamma_e$ must have a fixed sign for all edges $e$. Finally we easily see that this occurs when the intervals $I_e$ are chosen as given in the theorem.
\end{proof}




\bibliography{TopologicalStatesInTheKuramotoModel}

\begin{thebibliography}{10}

\bibitem{MR3513871}
Jared~C. Bronski, Lee DeVille, and Timothy Ferguson.
\newblock Graph homology and stability of coupled oscillator networks.
\newblock {\em SIAM J. Appl. Math.}, 76(3):1126--1151, 2016.

\bibitem{Buck1988SynchronousRhythmicFlashing}
J.~Buck.
\newblock {Synchronous Rhythmic Flashing of Fireflies. II.}
\newblock {\em The Quarterly Review of Biology}, 63(3):265--289, 1988.

\bibitem{MR3388491}
Taras Girnyk, Martin Hasler, and Yuriy Maistrenko.
\newblock Multistability of twisted states in non-locally coupled
  {K}uramoto-type models.
\newblock {\em Chaos}, 22(1):013114, 10, 2012.

\bibitem{PhysRevA.52.4089}
S.~Yu. Kourtchatov, V.~V. Likhanskii, A.~P. Napartovich, F.~T. Arecchi, and
  A.~Lapucci.
\newblock Theory of phase locking of globally coupled laser arrays.
\newblock {\em Phys. Rev. A}, 52:4089--4094, Nov 1995.

\bibitem{MR762432}
Y.~Kuramoto.
\newblock {\em Chemical oscillations, waves, and turbulence}, volume~19 of {\em
  Springer Series in Synergetics}.
\newblock Springer-Verlag, Berlin, 1984.

\bibitem{MR3129708}
Georgi~S. Medvedev.
\newblock Small-world networks of {K}uramoto oscillators.
\newblock {\em Phys. D}, 266:13--22, 2014.

\bibitem{MR3415044}
Georgi~S. Medvedev and Xuezhi Tang.
\newblock Stability of twisted states in the {K}uramoto model on {C}ayley and
  random graphs.
\newblock {\em J. Nonlinear Sci.}, 25(6):1169--1208, 2015.

\bibitem{MR3600363}
Georgi~S. Medvedev and J.~Douglas Wright.
\newblock Stability of {T}wisted {S}tates in the {C}ontinuum {K}uramoto
  {M}odel.
\newblock {\em SIAM J. Appl. Dyn. Syst.}, 16(1):188--203, 2017.

\bibitem{peskin75}
Charles~S. Peskin.
\newblock {\em Mathematical aspects of heart physiology}.
\newblock Courant Institute of Mathematical Sciences, New York University, New
  York, NY, USA, 1975.

\bibitem{MR1783382}
Steven~H. Strogatz.
\newblock From {K}uramoto to {C}rawford: exploring the onset of synchronization
  in populations of coupled oscillators.
\newblock {\em Phys. D}, 143(1-4):1--20, 2000.
\newblock Bifurcations, patterns and symmetry.

\bibitem{MR2878025}
Richard Taylor.
\newblock There is no non-zero stable fixed point for dense networks in the
  homogeneous {K}uramoto model.
\newblock {\em J. Phys. A}, 45(5):055102, 15, 2012.

\bibitem{MR2220552}
Daniel~A. Wiley, Steven~H. Strogatz, and Michelle Girvan.
\newblock The size of the sync basin.
\newblock {\em Chaos}, 16(1):015103, 8, 2006.

\end{thebibliography}
\bibliographystyle{plain}



\end{document}